\documentclass[12pt, a4paper]{article}
\usepackage[hmargin=2.5cm, vmargin=2cm, includefoot, twoside]{geometry}
\usepackage[backref=page,colorlinks,bookmarksopen=true]{hyperref}
\usepackage[latin1]{inputenc}  
\usepackage[T1]{fontenc}       
\usepackage[english]{babel}
\usepackage{amsmath}
\usepackage{amssymb}
\usepackage{amsthm}
\usepackage{latexsym}
\usepackage{mathrsfs}
\usepackage{graphics, graphicx}
\newtheorem{prop}{Proposition}[section]
\newtheorem{thm}[prop]{Theorem}
\newtheorem{cor}[prop]{Corollary}
\newtheorem{lem}[prop]{Lemma}

\theoremstyle{definition}

\newtheorem*{remark}{Remark}
\theoremstyle{remark}

\newcommand{\vareps}{\varepsilon}

\newcommand{\N}{\mathbb{N}}
\newcommand{\R}{\mathbb{R}}

\newcommand{\la}{\langle}
\newcommand{\ra}{\rangle}
\newcommand{\inv}{^{-1}}

\newcommand{\TLF}{topologically locally finite}
\newcommand{\CAT}[1]{\mathsf{CAT}(#1)}
\def\bs#1.{
              \def\temp{#1}
              \ifx\temp\empty
                   \mathcal{B}
              \else
                   \mathcal{B}(#1)
              \fi
}
\makeatletter

\@addtoreset{equation}{section} \makeatother
\DeclareMathOperator{\proj}{proj} 
\DeclareMathOperator{\Stab}{Stab}  
 \DeclareMathOperator{\Ker}{Ker}

\DeclareMathOperator{\Isom}{Isom}

\DeclareMathOperator{\GeomDim}{GeomDim}
\DeclareMathOperator{\LF}{Rad_{\mathscr{L\!F\!}}}

\DeclareMathOperator{\bd}{\partial_\infty}\DeclareMathOperator{\bdfine}{\partial_\infty^{\mathrm{fine}}\!}
\DeclareMathOperator{\bdcomb}{\partial_\infty^{\mathrm{comb}}\!}

\long\def\symbolfootnote[#1]#2{\begingroup%
\def\thefootnote{\fnsymbol{footnote}}\footnote[#1]{#2}\endgroup}
\begin{document}
\title%
{Amenable groups and Hadamard spaces with a totally disconnected isometry group}
\author{Pierre-Emmanuel Caprace\thanks{F.N.R.S. research fellow, partially supported by the Wiener-Anspach Foundation.}}
\date{1st May 2007}

\maketitle

\begin{abstract}
Let $X$ be a locally compact Hadamard space and $G$ be a totally disconnected group acting continuously,
properly and cocompactly on $X$. We show that a closed subgroup of $G$ is amenable if and only if it is
(topologically locally finite)-by-(virtually abelian). We are led to consider a set $\bdfine X$ which is a
refinement of the visual boundary $\bd X$. For each $x \in \bdfine X$, the stabilizer $G_x$ is amenable.
\end{abstract}

\newcommand{\subjclass}[1]{\symbolfootnote[0]{\noindent AMS classification numbers (2000):~#1.}}
\newcommand{\keywords}[1]{\symbolfootnote[0]{\emph{Keywords}:~#1}}

\subjclass{20F65; 20E42, 20F50, 43A07} 
\keywords{amenable group, $\CAT0$ space, totally disconnected group, locally finite group.}
\section{Introduction}

The class of amenable locally compact groups enjoys remarkable closure properties with respect to algebraic
operations, such as taking quotients or closed subgroups, or forming group extensions. However, despite of this
nice algebraic behaviour, the interaction between the amenability of a given group and the algebraic structure
of that group is still not completely understood. This is notably illustrated by the still unresolved problem to
show whether or not there exists an infinite finitely generated simple group which is amenable. On the other
hand, for some special classes of locally compact groups, the notion of amenability has a very well understood
algebraic interpretation. For example, it is known that a connected locally compact group is amenable if and
only if its solvable radical is cocompact \cite[Th.~3.8]{Paterson_amenability}. Therefore, understanding the
structure of amenable locally compact groups amounts to understand the structure of amenable locally compact
groups which are totally disconnected. The purpose of this paper is to show that in some cases, the relevant
algebraic property for amenable totally disconnected groups is the notion of {topological local finiteness}. A
subgroup $H$ of a topological group $G$ is called \textbf{topologically locally finite} if every finite subset
of $H$ topologically generates a compact subgroup of $G$. Basic facts on \TLF{} groups may be found in
Sect.~\ref{sect:topoLF} below. Here we merely mention a result of V.~Platonov \cite{Platonov65} (see
Proposition~\ref{prop:Platonov} below) which ensures that the class of \TLF{} groups is closed under group
extensions. In particular, any topological group $G$ possesses a \textbf{\TLF{} radical}, or
\textbf{LF-radical}, namely a maximal normal subgroup $N$ which is \TLF{} and such that $G/N$ has no nontrivial
normal \TLF{} subgroup. The LF-radical of $G$ is denoted by $\LF(G)$.

\medskip In this paper we focus on isometry groups of locally compact Hadamard spaces. Recall that a Hadamard
space is a complete $\CAT0$ space. Given a locally compact Hadamard space $X$, its isometry group $\Isom(X)$,
endowed with the topology of uniform convergence on compact subsets, is a locally compact second countable
group. The result of this paper is the following:
\begin{thm}\label{thm:StructureAmenable}
Let $X$ be a locally compact Hadamard space and $G$ be a totally disconnected group acting continuous, properly
and cocompactly on $X$. Then a closed subgroup $H < G$  is amenable if and only if $\LF(H)$ is open in $H$ and
the quotient   $H/\LF(H)$ is virtually abelian.
\end{thm}

\begin{cor}\label{cor:AmenableSimple}
Let $\Gamma$ be a finitely generated simple subgroup of $G$. If $\Gamma$ is contained in an amenable subgroup of
$G$, then it is finite.
\end{cor}

The following corollary is a generalization to amenable subgroups of the so called Solvable Subgroup Theorem for
$\CAT0$ groups \cite[Ch.~II, Th.~7.8]{BH99}:
\begin{cor}\label{cor:SolvSubgp}
Let $\Gamma$ be a group acting properly discontinuously and cocompactly on a complete $\CAT0$ space $X$. Then
any amenable subgroup of $\Gamma$ is virtually abelian and finitely generated.
\end{cor}

We refer to the introduction and reference list of \cite{AB98} for a historical background on amenability in the
geometrical context of non-positive curvature. The proof of Theorem~\ref{thm:StructureAmenable} is based on the
one hand, on obstructions for amenable groups to act by isometries on Hadamard spaces established by S.~Adams
and W.~Ballmann \cite{AB98} (see Proposition~\ref{prop:AdamsBallmann} below) and, on the other hand, on an
elementary  construction which associates to every point $\xi$ of the visual boundary of any $\CAT0$ space $X$
another $\CAT0$ space $X_\xi$. This construction is described in Sect.~\ref{sect:X_xi} below; it was considered
before by F.~Karpelevi\v{c} \cite{Karpelivic} in the context of symmetric spaces, (implicitly) by F.~Bruhat and
J.~Tits \cite[Prop.~7.6.4]{BT72} in the context of Euclidean buildings, and then by B.~Leeb \cite[\S2.1.3]{Leeb}
in the context of general Hadamard spaces. As suggested in \cite{Leeb}, one may iterate this construction to
define a boundary $\bdfine X$ of a proper $\CAT0$ space $X$ of bounded geometry which refines the usual visual
boundary $\bd X$, in the sense that there is a canonical $\Isom(X)$--equivariant surjection $\bdfine X \to \bd
X$. Therefore, the set $\bdfine X$ is called the \textbf{refined visual boundary} of $X$. A generic point of
$\bdfine X$ is a sequence of the form $(\xi_1, \xi_2, \dots, \xi_k, x)$ such that $\xi_1 \in \bd X$, $\xi_{i+1}
\in \bd X_{\xi_1, \dots, \xi_i}$ for each $i=1, \dots, k-1$ and $x \in X_{\xi_1, \dots, \xi_k}$. When $X$ is of
bounded geometry, the maximal possible length of this sequence happens to be bounded above by a constant
depending only on $X$, see Corollary~\ref{cor:Bound_GeomDim} below. The following result provides a more
geometric description of amenable subgroups of $G$; modulo the main result of \cite{AB98}, it is essentially
equivalent to Theorem~\ref{thm:StructureAmenable}:


\begin{thm}\label{thm:RefinedBdry}
Let $X$ be a locally compact Hadamard space and $G$ be a totally disconnected group acting continuously,
properly and cocompactly on $X$. Any amenable subgroup of $G$ has a finite index subgroup which fixes a point in
$X \cup \bdfine X$. Conversely, given any point $x \in X \cup \bdfine X$, the stabilizer $G_x$ is amenable.
\end{thm}

It is likely that if $\Isom(X)$ is cocompact, then the full stabilizer $\Isom(X)_x$ of any point $x \in X \cup
\bdfine X$ is always amenable. In fact, this is already known if $X$ is $\CAT{-1}$ by a result of Sh.~Mozes and
M.~Burger \cite[Prop.~1.6]{BM96}.

We remark that for some specific examples of Hadamard spaces $X$, the hypothesis that $G$ is cocompact may be
relaxed. This is for example the case when $X$ is a (cellular) building of arbitrary type or a $\CAT0$ cube
complex. In that case, the hypothesis that $G$ is cocompact may be replaced by: $X$ is proper of bounded
geometry and $G$ acts properly by cellular isometries. All results stated above remain true in this context. An
important point in this case is the existence of a notion of combinatorial convexity, which is a handful
supplement of the notion of geodesic convexity. In fact, the specific combinatorial structure of $X$ is
inherited by the asymptotic space $X_\xi$: if $X$ is a building (resp. a cube complex), then so is $X_\xi$ for
each $\xi \in \bd X$. In fact, this may be used to define a boundary space in a purely combinatorial way which
is a quotient of $\bdfine X$. This boundary is called the \textbf{combinatorial boundary} and denoted $\bdcomb
X$. It remains true that for each $x \in \bdcomb X$, the stabilizer $G_x$ is amenable.

\subsection*{Acknowledgements}

A substantial part of this work was accomplished when I was visiting F.~Haglund at the Université de Paris XI in
Orsay; I would like to thank him for stimulating discussions. The results presented here were first discovered
in the special case when $X$ is a building. Their extensions to a higher level of generality benefited from
exchanges with N.~Monod, to whom I express my special gratitude.

\section{On topologically locally finite groups}\label{sect:topoLF}

Let $G$ be a topological group. A subgroup $H < G$ is called \textbf{\TLF{}} if the closure of every finitely
generated subgroup of $H$ is compact. It is readily seen that if $G$ itself is \TLF, then so is any subgroup and
any continuous quotient group. Moreover, we have:

\begin{lem}\label{lem:TLF_closure}
Let $G$ be a locally compact group and $H$ be a \TLF{} subgroup. Then the closure $\overline H$ is \TLF{}, and
$\overline H$ endowed with the induced topology is a \TLF{} group.
\end{lem}
\begin{proof}
Suppose that $G$ contains a dense subgroup $H$ which is \TLF. We must show that $G$ itself is \TLF. Let $C$ be a
relatively compact open neighborhood of the identity in $G$. Given $c_1, \dots, c_k \in G$, the subset $C_1 =
\bigcup_{i=0}^k c_i.C$, where $c_0 = 1$, is a relatively compact open neighborhood of the identity containing
$c_1, \dots, c_k$. We set $U = C_1 \cup C_1\inv$. It suffices to prove that the subgroup of $G$ generated by $U$
is compact. Note that this subgroup is open, hence closed.

Let now $y \in \overline{U^2}$. Then $y.U \cap U^2$ is a nonempty open set. Hence there exists $h \in H \cap y.U
\cap U^2$. Since $U = U\inv$ and $h \in y.U$, we have $y \in h.U$. Therefore, we deduce that $\overline{U^2}
\subset \bigcup_{h \in H \cap U^2} h.U$. Since $\overline{ U^2}$ is compact, there exist $h_1, \dots, h_n \in H$
such that $\overline{U^2} \subset \bigcup_{i=1}^n  h_i.U$. Let $K$ be a compact subgroup of $G$ containing
$\{h_1, \dots, h_n\}$. Then we have:
$$U^3 = U^2 \cdot  U \subset (K \cdot U) \cdot U = K \cdot U^2 \subset K \cdot K \cdot U = K \cdot U.$$
We obtain inductively that $U^n$ is contained in $K \cdot U$ for each $n$. Since $\la U \ra = \bigcup_{n >0}
U^n$ and since $K \cdot \overline U$ is compact, it follows that $\la U \ra$ is compact, as desired.
\end{proof}

By Zorn's lemma, any topological group $G$ possesses a maximal normal subgroup which is topologically locally
finite. It is called the \textbf{LF-radical} of $G$ and denoted $\LF(G)$. By Lemma~\ref{lem:TLF_closure}, the
LF-radical of a locally compact group is a closed subgroup. The following result was proven by V. Platonov
\cite[Th.~2]{Platonov65}:

\begin{prop}\label{prop:Platonov}
Let $G$ be a locally compact group and $N$ be a closed normal subgroup. If $N$ and $G/N$ are both \TLF{}, then
so is $G$.
\end{prop}

It follows from Proposition~\ref{prop:Platonov} that $\LF(G/\LF(G)) = \{1\}$ for any locally compact group $G$.
Another useful basic fact is the following:

\begin{lem}\label{lem:TLF:compact}
Let $G$ be a locally compact group. Then $G$ is \TLF{} if and only if every compact subset of $G$ is contained
in a compact subgroup.
\end{lem}
\begin{proof}
The `if' part is clear. We focus on the opposite implication and assume henceforth that $G$ is \TLF{}. Let $Q$
be a compact subset of $G$ such that $Q = Q\inv$. We must show that $Q$ is contained in a compact subgroup of
$G$. Up to replacing $G$ by the closed subgroup which is generated by $Q$, this amounts to showing that if $G$
is compactly generated, then it is compact. Let thus $U$ be a compact symmetric neighborhood of the identity
which generates $G$. There exist $g_1, \dots, g_n \in G$ such that $U^2 \subset \bigcup_{i = 1}^n g_i.U$. Now we
can conclude by the same argument as in the proof of Lemma~\ref{lem:TLF_closure}.
\end{proof}

\begin{cor}\label{cor:TLF=>amenable}
Let $G$ be a locally compact group which is \TLF. Then $G$ is amenable.
\end{cor}
\begin{proof}
Follows from Lemma~\ref{lem:TLF:compact} together with F\o{}lner's characterization of amenability.
\end{proof}

%

\section{On proper actions of totally disconnected groups on Hadamard spaces}

Let $X$ be a locally compact Hadamard space, namely a complete locally compact $\CAT0$ space. Let also $G$ be a
totally disconnected group acting continuously upon $X$. Any compact subgroup of $G$ fixes a point in $X$ by
\cite[Ch.~II, Cor.~2.8]{BH99}. Recall that $\Isom(X)$, endowed with the topology of uniform convergence on
compact subsets, is a locally compact group. In particular, if the $G$--action on $X$ is proper, then $G$ is
locally compact. The following basic fact will be useful:

\begin{lem}\label{lem:CompactOpen}
Assume that $G$ acts properly on $X$. Then every compact subgroup of $G$ is contained in a compact open subgroup
of $G$.
\end{lem}
\begin{proof}
Let $K$ be a compact subgroup of $G$. Since $G$ is locally compact, the set $\mathcal B(G)$ of compact open
subgroups is nonempty and endowed with a canonical metric structure. Furthermore the $G$--action on $\mathcal
B(G)$ by conjugation is continuous. In particular, the group $K$ has a bounded orbit in $\mathcal B(G)$. By
\cite[Prop.~5]{BaumgartnerWillis06}, this implies that $K$ has a fixed point in $\mathcal B(G)$. In other words
$G$ possesses a compact open subgroup $C$ normalized by $K$. Since $C$ is compact, it has a fixed point in $X$.
Moreover, since $K$ normalizes $C$, it stabilizes the fixed point set $X^C$ of $C$ in $X$. Since $C$ itself is
compact, it fixes a point in $X$, hence in the closed convex subset $X^C$ which is $K$--invariant. This shows
that $C$ and $K$ have a common fixed point in $X$, say $x$. Since the $G$--action on $X$ is proper, the
stabilizer $G_x$ is compact. Since it contains $C$ it is open. Thus $K$ is contained in the compact open
subgroup $G_x$.
\end{proof}

We say that the $G$--action is \textbf{smooth} if $G_x$ is open in $G$ for each $x \in X$.  The term
\emph{smooth} is borrowed from the representation theory of $p$-adic groups.

The following lemma, though elementary, is crucial to the proof of the main results:
\begin{lem}\label{lem:smooth}
Assume that $G$ acts properly on $X$. Let  $(x_n)_{n \geq 0}$ be a sequence of points of $X$ and $(\gamma_n)_{n
\geq 0}$ be a sequence of elements of $G$ such that the sequence $(\gamma_n. x_n)_{n \geq 0}$ has a subsequence
converging to some $c \in X$. Then we have the following:
\begin{itemize}
\item[(i)] There exists a sequence $(x'_n)_{n \geq 0}$ of points of $X$ such that, given any $g \in G$ with
$\lim_{n \to \infty} d(x_n, g.x_n) = 0$, we have $g. x'_n = x'_n$ for all but a finite number of indices $n \geq
0$.

\item[(ii)] Assume moreover that the $G$--action is smooth. Then, given any $g \in G$ such that $\lim_{n \to
\infty} d(x_n, g.x_n) = l$, there exists $k \in G$ such that $d(c, k.c) = l$ and that the set $\{ n \geq 0 \; |
\; \gamma_n g \gamma_n\inv \in k.G_c \}$ is infinite.
\end{itemize}
\end{lem}
\begin{proof}
Up to extracting, we may and shall assume that $\lim_{n \to \infty} \gamma_n.x_n = c$. Let $g \in G$ be such
that $\lim_{n \to \infty} d(x_n, g.x_n) = l$. We have
$$\lim_{n \to \infty} d(x_n, g.x_n) =\lim_{n \to \infty}
d(\gamma_n. x_n, (\gamma_n g \gamma_n\inv) \gamma_n .x_n) = l.$$
Therefore, it follows that $\lim_{n \to \infty}
d(c, \gamma_n g \gamma_n\inv.c) = l$. In particular, the set $\{\gamma_n g \gamma_n\inv\}_{n \geq 0}$ is
relatively compact in $G$. Hence, up to extracting, we may assume that the sequence $(\gamma_n g
\gamma_n\inv)_{n \geq 0}$ converges. By construction, its limit $k$ maps the point $c$ to a point $c'$ such that
$d(c, c')=l$.

Assume first that $l=0$. Thus $c= c'$. By Lemma~\ref{lem:CompactOpen}, there exists $x \in X$ such that $G_c
\subset G_x$ and $G_x$ is compact open. Since the sequence $(\gamma_n g \gamma_n\inv)_{n \geq 0}$ converges to
$k \in G_x$, it follows that the set $\{ n \geq 0 \; | \; \gamma_n g \gamma_n\inv \in G_{x} \}$ contains all
sufficiently large $n$. Now, setting $x'_n = \gamma_n\inv.x$, we obtain that $g$ fixes $x'_n$ for almost all
$n$. Thus (i) holds.

Assume now that $l$ is arbitrary and that $G$ acts smoothly. Then $G_c$ is open, hence so is the coset $k.G_c =
\{h \in G \;|\; h.c = c'\}$. Therefore, for all $n$ sufficiently large, we have $\gamma_n g \gamma_n\inv \in
k.G_c$ and (ii) holds.
\end{proof}

Recall that, given $\gamma \in G$, the displacement function of $\gamma$ is the map $d_\gamma : X \to \R_+ : x
\mapsto d(x, \gamma.x)$. Its infimum is denoted by $|\gamma|$ and is called the translation length of $\gamma$
in $X$.

Note that when $G$ is cocompact,  the existence of a sequence $(\gamma_n)_{n \geq 0}$ as in the lemma is
automatic. In particular, we obtain (see \cite[Ch.~II, Sect.~6.1--6.3]{BH99}):

\begin{cor}\label{cor:parabolic}
Assume that $G$ acts properly and cocompactly on $X$. Then every element $\gamma \in G$ with $|\gamma| = 0$ has
a fixed point in $X$, and the set $\{ |\gamma| \; | \; \gamma \in G\}$ of translation lengths of elements of $G$
is discrete at $0$. Furthermore, if the $G$--action is smooth, then it is semisimple: any element acts as an
elliptic or a hyperbolic isometry.
\end{cor}
\begin{proof}
Let $\gamma \in G$ and choose $x_n \in X$ so that $d(x_n, \gamma.x_n)$ tends to $|\gamma|$ as $n$ tends to
infinity. Since $X/G$ is compact, there exists $\gamma_n \in G$ such that $\{\gamma_n.x_n\}$ is relatively
compact in $X$. Thus, up to extracting, we may assume that $(\gamma_n.x_n)_{n \geq 0}$ converges to some $c \in
X$. If $|\gamma| = 0$, then Lemma~\ref{lem:smooth}(i) shows that $\gamma_n \gamma \gamma_n\inv$ is elliptic for
some $n$, hence so is $\gamma$. Similarly, if the $G$--action is smooth, Lemma~\ref{lem:smooth} shows that the
displacement function $d_\gamma$ attains its infimum $|\gamma|$.

Let now $(g_n)_{n \geq 0}$ be a sequence of elements of $G$ such that $|g_n|$ tends to~$0$ as $n$ tends to
infinity and assume in order to obtain a contradiction that $|g_n| > 0$ for all $n$. Since $X/G$ is compact, we
may and shall assume, up to replacing $g_n$ by a conjugate, that there exists $c \in X$, $r \in \R$ and $x_n \in
X$ such that $d(c, x_n)< r$ for all $n$ and that $d(x_n, g_n.x_n)$ tends to $0$ as $n$ tends to infinity. Up to
extracting, we may assume that the sequence $(x_n)_{n \geq 0}$ converges to some $x \in X$. Since $\{g_n\}_{n
\geq 0}$ is relatively compact in $G$, we may assume, up to a further extraction, that $(g_n)_{n \geq 0}$
converges to some $g \in G$. Clearly $g$ fixes $x$. By Lemma~\ref{lem:smooth}(i), this implies that $g_n$ is
elliptic for all $n$ sufficiently large. Thus $|g_n|= 0$, which is absurd.
\end{proof}

Recall from \cite[Th.~A]{Br99} that if $X$ is a $\CAT0$ cell complex with finitely many isometry types of cells,
and if the $G$--action is cellular, then it is semisimple and the set of translation lengths of elements of $G$
is discrete at $0$. Thus the hypothesis that $X/G$ is compact is superfluous in that special case. Note that $G$
is automatically smooth in this case.

We record the following observation:
\begin{lem}\label{lem:smooth:flats}
Assume that $G$ acts properly and smoothly on $X$. Let $F \subset X$ be a flat and let $\varphi : G_{\{F\}} \to
\Isom(F)$ be the homomorphism induced by the action of the stabilizer $G_{\{F\}}$ of $F$ upon $F$. Then
$\varphi(G_{\{F\}})$ is a discrete subgroup of $\Isom(F)$. In particular $G_{\{F\}}/\Ker \varphi$ is virtually
abelian and $G_{\{F\}}$ possesses a finite index subgroup which fixes a point in $\bd F$.
\end{lem}
\begin{proof}
Let $\Gamma = \varphi(G_{\{F\}})$. We must show that $\Gamma < \Isom(F)$ acts properly discontinuously on $F$.

Let $x \in F$ be any point. We may choose $n+1$ points $x_0, \dots, x_n$, where $n = \dim F$, in such a way that
the group $G_{\{F\}, x_0, \dots, x_n}$ fixes pointwise a neighborhood of $x$ in $F$. Therefore, the group
$G_{\{F\}, x_0, \dots, x_n}$ is contained in $\Ker \varphi$. Since $G_{x_0, \dots, x_n}$ is an open subgroup of
the compact open subgroup $G_x$, it follows that the index of $G_{x_0, \dots, x_n}$ in $G_x$ is finite. In
particular, for each $x \in F$, the index of $\Ker \varphi$ in $G_{\{F\}, x}$ is finite or, in other words, for
each $x \in F$, the stabilizer $\Gamma_x$ is finite.

Suppose now that the $\Gamma$--action on $F$ is not properly discontinuous. Then there exist $x_0 \in F$ and $r
\in \R_+$ such that the set $S_\Gamma = \{\gamma \in \Gamma \; | \; d(\gamma. x_0, x_0)\}$ is infinite. Since
$\Gamma_x$ is finite for each $x \in X$, it follows that the set $S_0 = \{\gamma.x_0 \; | \; \gamma \in
S_\Gamma\}$ is infinite. Let $x_1 \in F$ be a cluster point of $S_0$. Let also $(g_n)_{n \geq 0}$ be a sequence
of elements of $G_{\{F\}}$ such that $\lim g_n.x_0 = x_1$ and that $g_m.x_0 \neq g_n.x_0$ for $m \neq n$. Since
$\{g_n\}_{n \geq 0}$ is relatively compact, we may assume that $(g_n)_{n \geq 0}$ converges to some $g \in G$
such that $g.x_0 = x_1$. Since $G_{x_0}$ is open in $G$, so is $g.G_{x_0}$. Therefore, we have $g_n.x_0 = x_1$
for all sufficiently large $n$. This contradicts the fact that $g_m.x_0 \neq g_n.x_0$ for $m \neq n$. Thus
$\Gamma$ is a discrete subgroup of $\Isom(F)$.

The fact that $\Gamma$ is virtually abelian now follows from \cite[Cor.~4.1.13]{Thurston97}. It remains to show
that $\Gamma$ has a finite index subgroup which fixes an element in the sphere at infinity $\bd F$. This is
trivial if $\Gamma$ is finite. If $\Gamma$ is infinite, then there exists an element $\gamma \in \Gamma$ which
acts as a hyperbolic isometry on $F$. Some power of $\gamma$ is centralized by a finite index subgroup $\Gamma_0
< \Gamma$. Therefore, the group $\Gamma_0$ fixes the unique attractive fixed point of $\gamma$ in $\bd F$.
\end{proof}

\section{Projective limits of horoballs: the space $X_\xi$}\label{sect:X_xi}

The purpose of this section is to study the main geometrical tool of this paper. In the first 
subsection, we collect some subsidiary facts on metric geometry.

\subsection{On metric spaces of bounded geometry}

Let $(X,d)$ be any metric space. Given $\vareps >0$, a subset $N \subset X$ is called \textbf{$\vareps$--sparse}
if $d(x, x') \geq \vareps$ for all $x \neq x' \in  N$. Note that a $\vareps$--sparse subset is discrete; in
particular, if it is contained in a compact subset, then it is finite. Given a subset $C \subset X$, we denote
by $n_\vareps(C)$ the maximal cardinality of a $\vareps$--sparse subset of $C$. Note that if $n_\vareps(C)$ is
finite, then a $\vareps$--sparse subset $N \subset C$ of maximal possible cardinality is necessarily
$\vareps$--dense: every point of $C$ is at distance less than $\vareps$ from some point of $N$. Given $r>0$ and
$\vareps >0$, we also set
$$n_{r, \vareps}(X) = \sup_{x \in X} n_\vareps(B(x, r)),$$
where $B(x, r)$ denotes the open ball of radius $r$ centered at $x$.

We say that the metric space $(X, d)$ is \textbf{of bounded geometry} if for all $r > \vareps >0$, one has
$n_{r, \vareps}(X) < \infty$. We record some elementary facts for later references:

\begin{lem}\label{lem:BoundedGeom}
We have the following:
\begin{itemize}
\item[(i)] If $(X, d)$ is complete and of bounded geometry, then it is proper, i.e. any closed ball is compact.

\item[(ii)] If $(X, d)$ is locally compact and $X/\Isom(X)$ is compact, then $X$ is of bounded geometry.
\end{itemize}
\end{lem}
\begin{proof}
(i). Follows from the characterization of compact metric spaces as those metric spaces which are complete and
totally bounded. 
The argument goes as follows. Let $B$ be a closed ball in $X$ and $S$ be an infinite set of points of $B$. Since
$X$ is of bounded geometry, the ball $B$ can be covered by a finite number of balls of radius $1$. Thus there
exists $b_0 \in B$ such that the ball $B(b_0, 1)$ contains an infinite subset of $S$. Repeating this argument
inductively, we construct a sequence $(b_n)_{n \geq 0}$ of points of $B$ such that $B(b_n, 2^{-n})$ contains an
infinite subset of $S$ and that $b_{n+1} \in B(b_{n}, 2^{-n})$. In particular the sequence $(b_n)_{n \geq 0}$ is
Cauchy. Let $b$ denote its limit. Clearly $b$ is a cluster point of $S$. Hence $B$ is compact.

\medskip \noindent The proof of (ii) is a standard exercise and will be omitted here.
\end{proof}

%
%

\subsection{The space $X_\xi$ and the refined boundary $\bdfine X$}

Let $X$ be any $\CAT0$ space. Given any point $\xi \in \partial_\infty X$ in the visual boundary of $X$, we now
describe a canonical construction which attaches a $\CAT0$ space $X_\xi$ to $\xi$. Any closed horoball centered
at $\xi$ is a closed convex subset of $X$. The collection of all of these horoballs form a chain of subspaces of
$X$. Endowing this chain with the orthogonal projections, we obtain a projective system of $\CAT0$ spaces. By
definition, the space $X_\xi$ is the metric completion of the projective limit of this system. Note that the
projective limit itself need not be complete even if $X$ is so; it is therefore important to take a completion
since we want to deal with Hadamard spaces. The space $X_\xi$ is endowed with a canonical surjective projection
$$\pi_\xi : X \to X_\xi$$
induced by the orthogonal projections onto horoballs. Note that $\pi_\xi$ is $1$--Lipschitz: it does not
increase distances.

There is a more down-to-earth description of $X_\xi$ which goes as follows. Let $X^*_\xi$ be the set of all
geodesics rays of $X$ which point toward $\xi$. The set $X^*_\xi$ is endowed with a pseudo-distance defined by:
$$d(\rho, \rho') = \inf_{t, t' \geq 0} d(\rho(t), \rho'(t')).$$
The space $X_\xi$ is the completion of the quotient of $X^*_\xi$ be the relation which identifies two rays at
distance~$0$, namely two rays which are strongly asymptotic. It is readily verified that this construction
yields the same space as the preceding one. Note that $X_\xi$ need not be locally compact, even if $X$ is so.

The fact that $\pi_\xi$ does not increase distances yields the following:
\begin{lem}\label{lem:eps-sparse}
Let $\xi \in \bd X$ and $r, \vareps > 0$ be positive numbers. Let $x_0, x_1, \dots, x_n \in X_\xi$ be such that
$d(x_0, x_i) < r$ for each $i$ and that the set $\{x_1, \dots, x_n\}$ is $\vareps$--sparse. Then there exist
$y_0, y_1, \dots, y_n, y_{n+1} \in X$ such that $d(y_0, y_i) < r$ for each $i$ and that the set $\{y_1, \dots,
y_n, y_{n+1}\}$ is $\vareps$--sparse.
\end{lem}
\begin{proof}
Let $\rho_0, \rho_1, \dots, \rho_n : \R_+ \to X$ be geodesic rays which are representatives of $x_0,$ $x_1,
\dots, x_n$ respectively. Note that for all $i =0, \dots, n$ and $t \in \R_+$, we have $\pi_\xi(\rho_i(t))=x_i$.
Let $R_0 = \rho_0(\R_+)$. By definition, for each $i =1, \dots, n$ there exists $t_i \in \R_+$ such that
$d\big(\rho_i(t_i), \proj_{R_0}(\rho_i(t_i))\big)<r$. Here $\proj$ denotes the orthogonal $\CAT0$ projection map
\cite[Ch.~II, Prop.~2.4]{BH99}. Let now $H$ be a closed horoball centered at $\xi$, whose radius is sufficiently
small so that $\{\rho_i(t_i), \proj_{R_0}(\rho_i(t_i)) \; | \; i=1, \dots, n\} \cap H = \varnothing$ and that
$\rho_0(\vareps)$ does not belong to $H$ either. Set $y_i = \proj_{H}(\rho_i(0))$ for each $i = 0, \dots, n$;
this makes sense in $H$ is closed and convex. Note that $\proj_{H}(\rho_i(0)) = \proj_{H}(\rho_i(t_i))$ for all
$i>0$. Therefore, we have
$$d(y_i, y_0) \leq d\big(\rho_i(t_i), \proj_{R_0}(\rho_i(t_i))\big) <r$$
for each $i=1, \dots, n$ since $\proj_H$ does not increase distances. Note also that the set $\{y_1, \dots,
y_n\}$ is $\vareps$--sparse since $\pi_\xi$ does not increase distances and since $\{x_1, \dots, x_n\}$ is
$\vareps$--sparse.

It remains to define $y_{n+1}$. To this end, let $t_0 \in \R_+$ be the unique real such that $\rho_0(t_0)=y_0$.
We set $y_{n+1} = \rho_0(t_0 - \vareps)$. Thus $d(y_0, y_{n+1}) = \vareps$. Since $\proj_H(y_{n+1})= y_0$, we
have $d(h, y_{ n+1}) \geq \vareps$ for all $y \in H$. In particular, the set $\{y_1, \dots, y_n, y_{n+1}\}$ is
$\vareps$--sparse. Finally, since $\vareps < r$, we have $d(y_0, y_{n+1})< r$ as desired.
\end{proof}

The following proposition collects some of the basic properties of $X_\xi$:

\begin{prop}\label{prop:X_xi}
Let $\xi \in \partial_\infty X$. We have the following:
\begin{itemize}
\item[(i)] $X_\xi$ is a complete $\CAT0$ space.

\item[(ii)] There is a canonical continuous homomorphism $\varphi_\xi : \Isom(X)_\xi \to \Isom(X_\xi)$, where
$\Isom(X)$ and $\Isom(X_\xi)$ are endowed with the topology of uniform convergence on compact subsets.

\item[(iii)]  If $X$ is proper and of bounded geometry, then so is $X_\xi$.
\end{itemize}
\end{prop}
\begin{proof}
(i). Follows immediately from the definition in terms of horoballs. For another argument using the alternative
construction of $X_\xi$, see B.~Leeb \cite[Proposition~2.8]{Leeb}.

\medskip \noindent
(ii). The map $\varphi_{\xi}$ is defined by:
$$\varphi_\xi(g).\pi_\xi(x) = \pi_\xi(g.x). $$
It is immediate from the definition that it is a homomorphism. Assume in order to obtain a contradiction that
$\varphi_\xi$ is not continuous. Then it is not continuous at $1$. Thus there exists a compact subset $C \subset
X_\xi$, a real $\vareps > 0$, a sequence $(y_n)_{n \geq 0}$ of points of $C$ and a sequence $(g_n)_{n \geq 0}$
of elements of $\Isom(X)_\xi$ such that $\lim_{n \to \infty} g_n = 1$ and $d(\varphi_\xi(g_n).y_n, y_n) >
\vareps$ for each $n$. Let $D \subset C$ be a finite subset which is $\frac{\vareps}{3}$--dense in $C$. Let $D'
\subset X$ be a finite subset such that $\pi_\xi(D') = D$. Since $\lim_{n \to \infty} g_n = 1$ and since $D'$ is
finite, we have $d(g_n.x, x) \leq \frac{\vareps}{3}$ for all $x \in D'$ and all sufficiently large $n$. Since
$\pi_\xi$ does not increase distances, we deduce from the definition of $\varphi_\xi$ that
$d(\varphi_\xi(g_n).y, y) \leq \frac{\vareps}{3}$ for all $y \in D$ and all  sufficiently large $n$. Since $D$
is $\frac{\vareps}{3}$--dense in $C$, it finally follows that $d(\varphi_\xi(g_n).z, z) \leq \vareps$ for all $z
\in C$ and all  sufficiently large $n$. This is a contradiction.

Note that $\varphi_\xi$ need not be proper.

\medskip \noindent
(iii). By definition, the space $X_\xi$ is complete. In view of Lemma~\ref{lem:BoundedGeom}(i), it is proper
whenever it is of bounded geometry. The fact that it is of bounded geometry follows easily from
Lemma~\ref{lem:eps-sparse}.
\end{proof}

Important to us will be the fact that the length of a sequence $(\xi_1, \xi_2, \dots, \xi_k)$ such that $\xi_1
\in \bd X$ and $\xi_{i+1} \in \bd X_{\xi_1, \dots, \xi_i}$ for each $i=1, \dots, k-1$ may not be arbitrarily
large under suitable assumptions on $X$:

\begin{cor}\label{cor:Bound_GeomDim}
Let $X$ be a complete $\CAT0$ space of bounded geometry. Then there exists an integer $K \geq 0$ depending only
on $X$ such that, given any sequence $(\xi_1, \xi_2, \dots, \xi_k)$ with $\xi_1 \in \bd X$ and $\xi_{i+1} \in
\bd X_{\xi_1, \dots, \xi_i}$ for each $i=1, \dots, k-1$, the space $X_{\xi_1, \dots, \xi_k}$ is bounded whenever
$k = K$. In particular $\bd X_{\xi_1, \dots, \xi_k}$ is empty whenever $k = K$.
\end{cor}
\begin{proof}
Suppose that $X_{\xi_1, \dots, \xi_k}$ is of diameter $>r$. Then $X_{\xi_1, \dots, \xi_k}$ contains two
points at distance $r$ from one another. Applying Lemma~\ref{lem:eps-sparse} inductively, we construct a finite
subset $N \subset X$ of cardinality $k+2$ which is $r$-sparse and of diameter $\leq r + \vareps$, where $\vareps
>0$ is a fixed positive number (which may be chosen arbitrarily small). In particular, we obtain
$k+2 \leq n_{r+\vareps, r}(X)$. The desired result follows.
\end{proof}

\begin{remark}
Using results of B.~Kleiner \cite{Kl99}, it can be shown that if $X$ is complete and $\GeomDim(X_\xi) \geq n$,
then $\GeomDim(X) \geq n+1$. In particular, if $\GeomDim(X) $ is finite, then $\GeomDim(X_\xi) < \GeomDim(X)$.
Therefore, if $X$ is complete and $\GeomDim(X) $ is finite, then there exists a constant $K$ such that $\bd
X_{\xi_1, \dots, \xi_k}$ is empty whenever $k \geq K$. Note that a $\CAT0$ space $X$ such that $\bd X$ is empty
might be unbounded: for example take $X$ to be a metric graph which is a star with infinitely many branches of
finite length, such that the supremum of the length of the branches is infinite. Note also that the fact that
$X$ is of finite geometric dimension is unrelated to the local compactness of $X$. In particular, if $X$ is a
$\CAT0$ piecewise Euclidean cell complex with finitely many types of cells (such as a building \cite{Da98} or a
finite dimensional cube complex), then $\GeomDim(X)$ is finite but $X$ need not be locally compact.
\end{remark}

\medskip
We define the \textbf{refined visual boundary} $\bdfine X$ to be the set of all sequences
$$(\xi_1, \xi_2, \dots, \xi_k, x)$$
such that $\xi_1 \in \bd X$, $\xi_{i+1} \in \bd X_{\xi_1, \dots, \xi_i}$ for each $i=1, \dots, k-1$ and $x \in
X_{\xi_1, \dots, \xi_k}$. Given such a sequence $(\xi_1, \xi_2, \dots, \xi_k, x)$ in the refined boundary, we
define its \textbf{level} to be the number $k$. In order to associate a level to each point of $X \cup \bdfine
X$, we take the convention that points of $X$ are of level~$0$. Corollary~\ref{cor:Bound_GeomDim} gives
sufficient conditions on $X$ for the existence of an upper bound on the level of all points in $X \cup \bdfine
X$.

\subsection{Structure of the stabilizer of a point in the refined boundary}

Given a point $\xi \in \partial_\infty X$ and a base point $x \in X$, we let $b_{\xi, x} : X \to \R$ be the
Busemann function centered at $\xi$ such that $b_{\xi, x}(x)=0$. Recall that Busemann functions satisfy the
following cocycle identity for all $x, y, z \in X$:
$$b_{\xi, x}(y) - b_{\xi, x}(z) = b_{\xi, z}(y).$$
It follows that the mapping
$$\beta_\xi : \Isom(X)_\xi \to \R : g \mapsto b_{\xi, x}(g.x)$$
is independent of the point $x \in X$ and is a group homomorphism. It is called the \textbf{Busemann
homomorphism} centered at $\xi$.

\begin{prop}\label{prop:X_xi:smooth}
Let $X$ be a proper $\CAT0$ space and $G$ be a totally disconnected group acting continuously, properly and
cocompactly on $X$. Given $\xi \in \partial_\infty X$, we have the following:
\begin{itemize}
\item[(i)] Given any $x \in X_\xi$, the LF-radical $\LF(G_{\xi, x})$ is open in $G_{\xi, x}$; it coincides with
the kernel of $\beta_\xi : G_{\xi, x} \to \R$.


\item[(ii)] Let $K_\xi$ be the kernel of the restriction of $\varphi_\xi$ to $G_\xi$. Then $\LF(K_\xi)$ is open
in $K_\xi$; it coincides with the kernel of $\beta_\xi : K_{\xi} \to \R$. In particular, the group
$K_{\xi}/\LF(K_{\xi})$ is isomorphic to a subgroup of $\R$.

\item[(iii)] Let $(\xi_1, \xi_2, \dots, \xi_n, x) \in \bdfine X$ be a  point of level $n$ in the refined visual
boundary. Set $H = G_{ \xi_1, \dots, \xi_n, x}$. Then $\LF(H)$ is open in $H$, it contains all elements of $H$
which act as elliptic isometries on $X$ and, furthermore, $H /\LF(H)$ is abelian and torsion free. In particular
$H$ is amenable.
\end{itemize}
\end{prop}
\begin{proof}
Note that (i) is a special case of (iii). However, the proof of (iii) involves some technicalities which can be
avoided in the situation of (i). Therefore, in order to make the argument  more transparent, we prove (i)
separately.

\medskip \noindent
(i). Let $K_{\xi, x}$ denote the kernel of the restriction to $G_{\xi, x}$ of the Busemann homomorphism
$\beta_\xi$. Let $y \in X$ be such that $\pi_\xi(y)=x$ and $\rho = \rho_{\xi, y} : \R_+ \to X$ be the geodesic
ray pointing towards $\xi$ with origin $y$. Define $x_n = \rho(n)$ for each $n \in \N$. Since $G$ is cocompact,
there exists a sequence $(\gamma_n)_{n \geq 0}$ of elements of $G$ such that $(\gamma_n.x_n)_{n \geq 0}$
converges to some $c \in X$. Now, given any $g_1, \dots, g_k \in K_{\xi, x}$, we have $\lim_{n \to \infty}
d(x_n, g_i.x_n) = 0$ for each $i=1, \dots, k$. Therefore, applying Lemma~\ref{lem:smooth} inductively, we deduce
that there exists $n \in \N$ such that $g_i \in G_{\gamma_n\inv.c}$ for each $i =1, \dots, k$. In particular,
the set $\{g_1, \dots, g_k\}$ is contained in a compact subgroup of $G$. This shows that $K_{\xi, x}$ is \TLF.

Now, the inclusion $K_{\xi, x} \subset \LF(G_{\xi, x})$ is obvious. Conversely, given any element $g \in G_{\xi,
x}$ which does not belong to $K_{\xi, x}$, then $g$ is not elliptic, hence it is not contained in $\LF(G_{\xi,
x})$. Thus $ K_{\xi, x} = \LF(G_{\xi, x})$ as desired.

The fact that $K_{\xi, x}$ is open in $G_{\xi, x}$ is clear: by definition $G_{\xi, x}$ is closed and any
compact open subgroup of $G_{\xi, x}$ fixes a point in $X$, and is thus contained in $K_{\xi, x}$.

\medskip \noindent
(ii). By definition, we have $K_\xi= \bigcap_{x \in X_\xi} G_{\xi, x}$. Hence the desired assertion follows
from~(i).

\medskip \noindent
(iii). For each $i= 1, \dots, n$, let $\beta_{\xi_i} : G_{\xi_1, \dots, \xi_i} \to \R$ be the restriction of the
Busemann homomorphism centered at $\xi_i$. In particular, restricting further, one obtains a homomorphism
$\beta_{\xi_i} : H \to \R$. The direct product of these homomorphisms defines a homomorphism
$$\beta =
\beta_{\xi_1} \times \dots \times \beta_{\xi_n} : H \to \R^{n},$$ whose kernel is the subgroup $K =
\bigcap_{i=1}^n \Ker \beta_{\xi_i}$. Clearly $K$ contains all elements of $H$ which act as elliptic isometries
on $X$ (and hence on $X_{\xi_1, \dots, \xi_i}$ for each $i=1, \dots, n$). In particular it follows that $K$ is
open in $H$.

Our aim is to show that $K = \LF(H)$. We have just seen that $K$ contains all periodic elements of $H$. Thus the
inclusion $\LF(H) \subset K$ is clear. It remains to show that $K$ is \TLF.

For each $i = 1, \dots, n$, we define
$$\varphi_i = \varphi_{\xi_i} \circ \dots \circ \varphi_{\xi_1} \circ \varphi_{\xi_1} :
\Isom(X)_{\xi_1, \dots, \xi_{i}} \to \Isom(X_{\xi_1, \dots, \xi_{i}}).$$
Let $g_1, \dots, g_k$ be elements of $K$. By definition, there exists a sequence $(x_{n-1, m})_{m \geq 0}$ of
points of $X_{\xi_1, \dots, \xi_{n-1}}$ such that
$$\lim_{m \to \infty} d(\varphi_{n-1}(g_i).x_{n-1, m}, x_{n-1, m}) = 0$$
 for each $i=1, \dots, k$. Let now $\rho_{n-2,m} : \R_+ \to X_{\xi_1, \dots, \xi_{n-2}}$ be a geodesic ray
pointing towards $\xi_{n-1}$ such that $\pi_{\xi_{n-1}}(\rho_{n-2,m}(t))=x_{n-1, m}$ for each $t \in \R_+$.

For each $m$, we may choose a sufficiently large  $t_m \in \R_+$ in such a way that the sequence $(x_{n-2,m})_{m
\geq 0}$ defined by $x_{n-2, m} = \rho_{n-2,m}(t_m) \in X_{\xi_1, \dots, \xi_{n- 2}}$ satisfies the identity
$$\lim_{m \to \infty} d(\varphi_{n-2}(g_i).x_{n-2, m}, x_{n-2, m}) = 0$$
for each $i=1, \dots, k$.

Proceeding inductively, we construct in this way a sequence  $(x_{j, m})_{m \geq 0}$ of points of $X_{\xi_0,
\dots, \xi_j}$ such that
$$\lim_{m \to \infty} d(\varphi_{j}(g_i).x_{j, m}, x_{j, m}) = 0$$
for each $i=1, \dots, k$ and each $j = 1, \dots, n-1$. In a final further step, we then construct a sequence
$(x_m)_{m \geq 0}$ of points of $X$ such that
$$\lim_{m \to \infty} d(g_i.x_{m}, x_{m}) = 0$$
for each $i=1, \dots, k$. Now, it follows by the same arguments as in the proof of (i) that $\{g_1, \dots,
g_k\}$ is contained in a compact subgroup of $G$. Hence $K$ is \TLF{}, as desired.

The amenability of $H$ is now immediate from Corollary~\ref{cor:TLF=>amenable}.
\end{proof}

Note that the proof of Proposition~\ref{prop:X_xi:smooth}(iii) shows that $\LF(H)$ coincides with $\Ker
\beta|_H$, where $\beta = \beta_{\xi_1} \times \dots \times \beta_{\xi_n} : G_{\xi_1, \dots, \xi_n} \to \R^n$ is
the direct product of the Busemann homomorphisms centered at $\xi_i$ for $i= 1, \dots, n$.

\begin{lem}\label{lem:semisimple:inherited}
Let $X$ be a proper $\CAT0$ space and $G$ be a totally disconnected group acting continuously, properly and
cocompactly on $X$. Then, given any element $\gamma \in  \Ker \beta$, the respective translation lengths of
$\gamma$ in $X$ and in $X_{\xi_1, \dots, \xi_n}$ coincide. Furthermore, if the $G$--action is smooth, then the
action of $G_{\xi_1, \dots, \xi_n}$ upon $X_{\xi_1, \dots, \xi_n}$ is by semisimple isometries.
\end{lem}
\begin{proof}
Let $\xi \in \bd X$. Since $\pi_\xi$ does not increase distances, it is clear that the translation length
$|\gamma|$ of any element $\gamma \in \Isom(X)_\xi$ is bounded below by the translation length
$|\varphi_\xi(\gamma)|$  of $\varphi_\xi(\gamma)$ in $X_\xi$. Conversely, if $\gamma \in \Ker \beta$, then it is
easy to see that $|\gamma| \leq |\varphi_\xi(\gamma)|$.

It is clear that an elliptic isometry $\gamma \in \Isom(X)$ which fixes $\xi$ acts as an elliptic isometry upon
$X_\xi$. Suppose now that $\gamma \in \Isom(X)$ is hyperbolic and fixes $\xi$. Let $\lambda$ be an axis of
$\gamma$. If $\xi \in \bd \lambda$, then $\gamma$ is elliptic on $X_\xi$. Otherwise, it follows easily from
\cite[Ch.~II, Prop.~9.8 and Cor.~9.9]{BH99} that $\lambda$ bounds a Euclidean half-plane $H$ such that $\xi \in
\bd H$. Moreover, one verifies immediately that the projection of $H$ to $X_\xi$ is an axis for $\gamma$, from
which it follows that $\gamma$ acts as a hyperbolic isometry on $X_\xi$. Note moreover that
$\beta_\xi(\gamma)=0$ if and only if $\xi$ is the middle point of $\bd H$, where $\beta_\xi$ denotes the
Busemann homomorphism centered at $\xi$.

Now, if the $G$--action is smooth, the fact that the $G_{\xi_1, \dots, \xi_n}$ upon $X_{\xi_1, \dots, \xi_n}$ is
semisimple follows from a straightforward induction on $n$, since we know by Corollary~\ref{cor:parabolic} that
the $G$--action upon $X$ is semisimple.
\end{proof}

\section{The structure of amenable subgroups}

The main tool in proving Theorem~\ref{thm:StructureAmenable} is provided by the obstructions for continuous
isometric actions of amenable groups on Hadamard spaces established in \cite{AB98}. Let us recall its precise
statement:
\begin{prop}\label{prop:AdamsBallmann}
Let $H$ be an amenable locally compact group acting continuously by isometries on a proper $\CAT0$ space $X$.
Then one of the following holds:
\begin{itemize}
\item[(i)] $H$ stabilizes a Euclidean flat in $X$;

\item[(ii)] $H$ fixes a point in $X \cup \partial_\infty X$.
\end{itemize}
\end{prop}
\begin{proof}
See \cite[Theorem]{AB98}.
\end{proof}

Before proceeding to the proof of the main results, we still need a subsidiary lemma:

\begin{lem}\label{lem:flats}
Let $X$ be a proper $\CAT0$ space and $G $  be a totally disconnected group acting continuously, properly and
cocompactly on $X$. Let $(\xi_1, \dots, \xi_n)$ be a sequence such that $\xi_1 \in \bd X$, $\xi_{i+1} \in \bd
X_{\xi_1, \dots, \xi_i}$ for each $i=1, \dots, n-1$ and let $F$ be a flat in $X_{\xi_1, \dots, \xi_n}$ (possibly
$n=0$ and $F \subset X$). Suppose that $H < G$ is a closed amenable subgroup which fixes $(\xi_1, \dots, \xi_n)$
and which stabilizes $F$. Then $H$ possesses a finite index subgroup which fixes a point in $F \cup \bd F$.
\end{lem}
\begin{proof}
%
As in the proof of Proposition~\ref{prop:X_xi:smooth}(iii), we let $\beta_{\xi_i} : G_{\xi_1, \dots, \xi_i} \to
\R$ be the restriction of the Busemann homomorphism centered at $\xi_i$ and
$$\beta = \beta_{\xi_1} \times \dots \times \beta_{\xi_n} : G_{\xi_1, \dots, \xi_n} \to \R^n$$
be the direct product of these Busemann homomorphisms. Let $R = \Ker \beta$.

By hypothesis, we have $H < \Stab_{G_{\xi_1, \dots, \xi_n}}(F)$. Thus there is a well defined homomorphism
$$\varphi : H \to \Isom(F).$$
Since $H$ is totally disconnected, it follows from \cite[Ch.~V, Th.~2]{MontgomeryZippin55} that $\varphi(H)$
(endowed with the quotient topology) is a discrete group. Since moreover $\varphi(H)$ is amenable and contained
in the real Lie group $\Isom(F)$, it follows from \cite[Th.~1]{TitsAlternative} that $\varphi(H)$ is virtually
solvable, hence virtually metabelian because $\Isom(F)$ is abelian-by-compact. Up to replacing $H$ by a finite
index subgroup, we may -- and shall -- assume henceforth that $\varphi(H)$ is metabelian.

We let $T$ denote the translation subgroup of $\Isom(F)$. Thus we have $[\varphi(H), \varphi(H)] \subset T$. On
the other hand, since $R = \Ker \beta$ contains the derived group $[H, H]$, we deduce that $[\varphi(H),
\varphi(H)] \subset T \cap \varphi(H \cap R)$. Now we distinguish several cases.

Assume first that $T \cap \varphi(H \cap R)$ is nontrivial. By Corollary~\ref{cor:parabolic} and
Lemma~\ref{lem:semisimple:inherited}, the set of translation lengths of elements of $R$ upon $X_{\xi_1, \dots,
\xi_n}$ is discrete at $0$. Therefore, it follows that $T \cap \varphi(H \cap R)$ is a discrete subgroup of $T$.
Let now $t \in T \cap \varphi(H \cap R)$ be a nontrivial element. Since $T \cap \varphi(H \cap R)$ is normal in
$\varphi(H)$ and since conjugate elements act with the same translation length, it follows from the discreteness
of $T \cap \varphi(H \cap R)$ in $T$ that $\varphi(H)$ possesses a finite index subgroup which centralizes $t$.
Since $t$ acts as a hyperbolic element, we deduce that its unique attractive fixed point in the sphere at
infinity $\bd F$ is fixed by a finite index subgroup of $H$. Hence we are done in this case.

We assume henceforth that $T \cap \varphi(H \cap R)$ is trivial. By the above, it follows that $\varphi(H)$ is
abelian. Suppose now $\varphi(H)$ contains an element $t'$ which acts as a hyperbolic element on $F$. Then
$\varphi(H)$ fixes the attractive fixed point of $t'$ in $\bd F$ and again we are done. Suppose finally that
every element in $\varphi(H)$ is elliptic. Since the fixed point set of an element in $\Isom(F)$ is a linear,
hence Euclidean, subspace, a straightforward induction on dimension  shows then that $\varphi(H)$ has a global
fixed point in $F$. This concludes the proof.
\end{proof}

We are now ready for the:
\begin{proof}[\textbf{Proof of Theorems~\ref{thm:StructureAmenable} and~\ref{thm:RefinedBdry}}]
Note that $X$ is complete and of bounded geometry, since $\Isom(X)$ is cocompact by hypothesis.

The fact that $G_x$ is (\TLF)-by-(virtually abelian) for each $x \in X \cup \bdfine X$ follows from
Proposition~\ref{prop:X_xi:smooth}(iii). Any such subgroup is amenable in view of
Corollary~\ref{cor:TLF=>amenable}.

Let now $H < G$ be a closed amenable subgroup. We want to show that $H$ possesses a finite index subgroup which
fixes an element of $X \cup \bdfine X$.

Assume that $H$ fixes no point in $X \cup \bd X$. In view of Proposition~\ref{prop:AdamsBallmann}, this implies
that $H$ stabilizes a flat $F \subset X$. By Lemma~\ref{lem:flats}, we deduce that $H$ possesses a finite index
subgroup which fixes a point in $\bd F$. This shows that in all cases $H$ possesses a finite index subgroup
$H_0$ which fixes a point $\xi_1 \in X \cup \bd X$.

If $\xi_1 \in X$ we are done. Otherwise $H_0$ acts on $X_{\xi_1}$. Assume that $H_0$ fixes no point in
$X_{\xi_1} \cup \bd X_{\xi_1}$. Then $H_0$ stabilizes a flat in $X_{\xi_1}$ and, by Lemma~\ref{lem:flats}, we
deduce that $H_0$ possesses a finite index subgroup $H_1$ which fixes a point $\xi_2$ in $X_{\xi_1} \cup \bd
X_{\xi_1}$. Again, if $\xi_2 \in X_{\xi_1}$ we are done. Otherwise $H_2$ acts on $X_{\xi_1, \xi_2}$.

Now we repeating this argument inductively. The process will stop after finitely many steps in view of
Corollary~\ref{cor:Bound_GeomDim}. Therefore, we obtain a point $(\xi_1, \dots, \xi_n, x) \in \bdfine X$ and a
finite index subgroup $H_n < H$ which is contained in $G_{\xi_1, \dots, \xi_n, x}$. By
Proposition~\ref{prop:X_xi:smooth}(iii), the latter subgroup is (\TLF{})-by-(abelian torsion free) and its
LF-radical is open.
\end{proof}

\begin{proof}[\textbf{Proof of Corollary~\ref{cor:AmenableSimple}}]
Let $\Gamma < G$ be a finitely generated simple subgroup which is contained in an amenable subgroup of $G$. In
view of  the characterization of amenability in terms of a fixed point property \cite[Th.~G.1.7]{BeHaVa07}, we
may and shall assume that $\Gamma$ is in fact contained in a closed amenable subgroup of $G$, say $H$. Let $H_0$
be its LF-radical. There are two cases.

Suppose first that $H_0 \cap \Gamma$ is trivial. Then $\Gamma$ injects in the quotient $H/H_0$, which is
virtually abelian. Since $\Gamma$ is simple and finitely generated, it must then be finite.

Suppose now that $H_0 \cap \Gamma$ is nontrivial. Then $\Gamma \subset H_0$. Therefore $\Gamma$ is contained in
a compact subgroup of $G$. Since any such subgroup is a profinite group, it follows that $\Gamma $ is residually
finite. Hence, since $\Gamma$ is simple, it must be finite.
\end{proof}

\begin{proof}[\textbf{Proof of Corollary~\ref{cor:SolvSubgp}}]
Let $H < \Gamma$ be an amenable subgroup. Let $F$ be its LF-radical. It is a discrete countable locally finite
group. In particular, it is a union of an ascending chain of finite subgroups of $\Gamma$. Since $\Gamma$ acts
geometrically on $X$, it follows from \cite[Ch.~II, Cor.~2.8]{BH99} that it has finitely many conjugacy classes
of finite subgroups. In particular $F$ is finite. Therefore, there exists a finite index subgroup $H_0 < H$
which centralizes $F$. By Theorem~\ref{thm:StructureAmenable}, the group $H/F$ is virtually abelian. Thus $H_0$
possesses a finite index subgroup $H_1$ such that the derived subgroup $[H_1, H_1]$ is contained in $F$. Since
any finitely generated group with a finite derived subgroup is virtually abelian \cite[Ch.~II, Lem.~7.9]{BH99},
it follows that any finitely generated subgroup of $H_1$ is virtually abelian. On the other hand, the group
$\Gamma$ satisfies an ascending chain condition for virtually abelian subgroups by \cite[Ch.~II, Th.~7.5]{BH99},
from which it finally follows that $H_1$ is virtually abelian and finitely generated and, hence, so is $H$.
\end{proof}

%

{\small \providecommand{\bysame}{\leavevmode\hbox to3em{\hrulefill}\thinspace}
\providecommand{\MR}{\relax\ifhmode\unskip\space\fi MR }
\providecommand{\MRhref}[2]{%
  \href{http://www.ams.org/mathscinet-getitem?mr=#1}{#2}
} \providecommand{\href}[2]{#2}

}

\noindent \textsc{D\'epartement de Math\'ematiques,
Universit\'e Libre de Bruxelles (U.L.B.), CP 216, Boulevard du Triomphe, 1050 Bruxelles, Belgique.}\\

\vspace{.25cm} \noindent
\emph{Current address:}\\
\textsc{Mathematical Institute, University of Oxford, 24--29 St Giles', Oxford OX1 3LB, United Kingdom.}\\
{\tt caprace@maths.ox.ac.uk}

\end{document}